\documentclass[12pt]{elsarticle}

\usepackage{amsthm,amsmath,amssymb,enumerate}
\usepackage{url}
\usepackage{bm}

\numberwithin{equation}{section}

\newtheorem{thm}{Theorem}[section]

\newtheorem{lem}[thm]{Lemma}
\newtheorem{lemma}[thm]{Lemma}

\newtheorem{corollary}[thm]{Corollary}
\theoremstyle{definition}

\newtheorem*{remark}{Remark}

\newtheorem{example}[thm]{Example}

\newcommand{\C}{\mathbb{C}}
\newcommand{\Z}{\mathbb{Z}}

\newcommand{\R}{\mathbb{R}}

\begin{document}
\begin{frontmatter}

\title{On new proper Jordan schemes related to quaternion and octonion algebras}

\author[Hanaki]{Akihide Hanaki}
\address[Hanaki]{Faculty of Science, Shinshu University, 3-1-1 Asahi, Matsumoto, 390-8621, Japan}
\ead{hanaki@shinshu-u.ac.jp}

\author[yoshi]{Masayoshi Yoshikawa}
\address[yoshi]{Department of Mathematics, Hyogo University of Teacher Education, 942-1 Shimokume, Kato, Hyogo, 673-1494, Japan}
\ead{myoshi@hyogo-u.ac.jp}

\begin{abstract}
  We present a construction of a Jordan scheme from an elementary abelian $2$-group of rank $n$ and a $\{1,-1\}$-matrix of order $2^n$ that satisfies a specified condition.
  We then prove that the orders of matrices with the specified condition are limited to $2, 4$ or $8$.
  Using these matrices, we construct essentially two new proper Jordan schemes of orders $16$ and $32$.
  Finally, we analyze the structures of the real adjacency Jordan algebras of these Jordan schemes and 
  prove that they are the first known examples whose real adjacency Jordan algebras admit simple components of type $\R \oplus_f \R^n$, namely non-Hermitian type.
\end{abstract}

\begin{keyword}
adjacency Jordan algebras \sep Jordan schemes \sep octonion algebra \sep quaternion algebra

\MSC[2020] 05E30, 05E99, 17C20
\end{keyword}

\end{frontmatter}

\section{Introduction}
Association schemes are main research objects in algebraic combinatorics, and has been investigated by many researchers.
Association schemes are defined as follows.

Let $X$ be a finite set, and let $M_X(\C)$ denote the set of complex matrices which rows and columns indexed by $X$.
Let $S = \{A_0 = I, A_1, \dots, A_d\} \subset M_X(\C)$ be a set of $\{0, 1\}$-matrices with the following conditions,
where $I$ denotes the identity matrix.
\begin{enumerate}
  \item[(1)] $\sum_{i=0}^{d} A_i = J$, where $J$ denotes the all-one matrix.
  \item[(2)] For any $i \in \{0, 1, \dots, d\}$, there exists $i^{*} \in \{0,1,\dots, d\}$ such that ${}^{t} A_{i} = A_{i^{*}}$, where ${}^t M$ denotes the transpose of the matrix $M$.
  \item[(3)] For $i, j, k \in \{0, 1, \dots, d\}$, there exists a non-negative integer $p_{ij}^{k}$ such that $A_i A_j = \sum_{k = 0}^{d} p_{ij}^{k} A_k$.
\end{enumerate}
Then the pair $(X, S)$ is called an \emph{association scheme}.
An association scheme is called \emph{symmetric} if all matrices of $S$ are symmetric, and \emph{commutative} if $A_i A_j = A_j A_i$ for any $i, j \in \{0,1,\dots, d\}$.

In \cite{cameron}, Cameron introduced the concept of {\it Jordan schemes}. He defined Jordan schemes by replacing the ordinary matrix product with the Jordan product in condition (3) of the definition of symmetric association schemes,
where the Jordan product $A * B$ of two matrices $A$ and $B$ of the same order is defined by
\[A * B := \dfrac{AB + BA}{2}.\]

Let $(X,S)$ be an association scheme. 
For any element $A_i \in S$, we define $\tilde{A_i}$ by 
\[\tilde{A_i} := \begin{cases}
  A_i & \text{ if $A_i = {}^t A_i$,} \\
  A_i + {}^t A_i & \text{ otherwise.}
\end{cases}\]
We set $\tilde{S} := \{ \tilde{A_i} : A_i \in S\}$.
Then the pair $(X, \tilde{S})$ is called the \emph{symmetrization} of the association scheme $(X,S)$.
It is well known that the symmetrization of any commutative association scheme is also an association scheme, whereas the symmetrization of a noncommutative association scheme is not necessarily an association scheme but is a Jordan scheme.
Thus a Jordan scheme is called \emph{improper} or \emph{nonproper} if it is the symmetrization of an association scheme, and \emph{proper} otherwise.
Cameron posed the question of whether proper Jordan schemes exist.
In \cite{kmr}, Klin, Muzychuk and Reichard gave an affirmative answer to this question.
In this paper, we introduce further examples of proper Jordan schemes and study the structures of their adjacency Jordan algebras over the field $\R$ of real numbers.
We show that these Jordan schemes furnish the first known examples whose real adjacency Jordan algebras admit simple components of type $\R \oplus_f \R^n$, namely non-Hermitian type.
Moreover, we prove that our construction yields exactly two proper Jordan schemes.

This paper is organized as follows.
In Section 2 we provide the definitions and the basic properties of Jordan schemes and Jordan algebras.
In Sections 3 and 4 we introduce a construction of a proper Jordan scheme from an elementary abelian $2$-group and a $\{1, -1\}$-matrix satisfying a specified condition.
We then obtain two new proper Jordan schemes.
In Section 5 we determine the structures of the adjacency Jordan algebras of two new proper Jordan schemes over the field $\R$ of real numbers.

\section{Jordan schemes and Jordan algebras}
For square matrices $A$ and $B$ of the same order, we define 
\[A * B := \dfrac{AB + BA}{2}\]
and call this the \emph{Jordan product}.
We remark that the Jordan product is not associative, in general.

It is easy to check the following lemma.
\begin{lemma} \label{AB=BA}
  \begin{enumerate}
    \item $A*B = B*A$.
    \item If $AB = BA$, then $A*B = AB$.
  \end{enumerate}  
\end{lemma}

Let $X$ be a finite set and let $M_X(\C)$ denote the set of complex matrices which rows and columns indexed by $X$. 
Let $S = \{A_0 = I, A_1, \dots, A_d \} \subset M_{X}(\C)$ be a set of symmetric $\{0,1\}$-matrices with the following conditions, where $I$ denotes the identity matrix.
\begin{enumerate}
  \item $\sum_{i=0}^{d} A_i = J$, where $J$ denotes the all-one matrix.
  \item For $i, j, k \in \{0, 1, \dots, d\}$, there exists a rational number $p_{ij}^{k} \in \frac{1}{2}\Z$ such that 
  \[A_i * A_j = \sum_{k = 0}^{d} p_{ij}^{k} A_k.\]
\end{enumerate}
In this case, we call the pair $(X, S)$ a \emph{Jordan scheme}; for simplicity, we may also refer to $S$ itself as a Jordan scheme.
From Lemma \ref{AB=BA}(2), all symmetric association schemes are Jordan schemes.

We refer to \cite{jacobson} and \cite{minnesota} on Jordan algebras.
In this paper, we consider only Jordan algebras over the field $\R$ of real numbers.

Let $A$ be an $\R$-vector space together with a bilinear multiplication $*$ which satisfies the following: For any $a, b \in A$,
\begin{enumerate}
 \item $a*b = b*a$,
 \item $a^{*2}*(b*a) = (a^{*2}*b)*a$, where $a^{*2} = a*a$.
\end{enumerate}
Then $A$ is called a {\it Jordan algebra}.

Let $B$ be an associative algebra over $\R$.
Then the vector space $B$, equipped with the Jordan product 
\[a * b = \frac{1}{2} (ab+ba), \] 
is a Jordan algebra, where $ab$ denotes the product in the associative algebra.
This Jordan algebra $(B, *)$ is denoted by $B^+$.
A Jordan algebra $A$ is called \emph{special} if $A$ is isomorphic to a subalgebra of $B^+$ for some associative algebra $B$.

In the following, let $A$ be a Jordan algebra.

A subset $I$ of $A$ is called an \emph{ideal} of $A$ if
\begin{enumerate}
	\item $I$ is a subspace of the vector space $A$,
	\item $a * x \in I$ for all $a \in A, x \in I$.
\end{enumerate}

Two elements $a, b \in A$ are said to \emph{commute} if for every $x \in A$,
\[a * (b * x) = b * (a * x). \]
We define the \emph{center} $Z(A)$ of $A$ by 
\[Z(A) := \{ a \in A : a, b \text{ commute for any } b \in A \}.\]

For any $a \in A$, there is a linear transformation $L(a)$, which is linear in $a$, such that $a * b = L(a)b$ for any $b \in A$.
For $a, b \in A$, we define
\[ \tau(a, b) = \mathrm{tr} L(a*b).\]
Then $\tau(a,b)$ is a symmetric bilinear form on the vector space $A$.
We define the \emph{radical} of the Jordan algebra $A$ by
\[rad A := \{ a \in A : \tau(a, x) = 0 \text{ for every $x \in A$}\}.\]
We call the Jordan algebra $A$ \emph{semisimple} if $rad A = \{0\}$, i.e. if $\tau(a,b)$ is non-singular.
A Jordan algebra is called \emph{formally real} if $a^{*2} + b^{*2} = 0$ implies $a=b=0$.
Any formally real Jordan algebra is semisimple (See \cite[Corollary 5 in Chapter VI]{minnesota}).

Let $(X, S)$ be a Jordan scheme and let $S = \{A_{0}=I, A_{1}, \dots, A_{d}\}$.
Then $\R S := \bigoplus_{i = 0}^{d} \R A_{i}$ is closed under the Jordan product, $A_{i} * A_{j} = \frac{1}{2} \left( A_{i} A_{j} + A_{j} A_{i}\right)$.
Thus, $\R S$ forms a special Jordan algebra, called the (real) {\it adjacency Jordan algebra} of the Jordan scheme $(X,S)$.

\begin{thm} \cite[p.325]{MRK}
	Let $\R S$ be the adjacency Jordan algebra of a Jordan scheme $(X,S)$. 
	Then $\R S$ is formally real and so semisimple.
\end{thm}

A Jordan algebra $A$ is said to be the {\it direct sum} of the subalgebras $I_{1}, I_{2}, \dots, I_{r}$, symbolically $A = I_{1} \oplus I_{2} \oplus \cdots \oplus I_{r}$, 
if $A$ is the direct sum of the vector spaces $I_{1}, I_{2}, \dots, I_{r}$, and if $I_{i} * I_{j} = \{0\}$ holds for $i \neq j$.
Each subalgebra $I_{i}$ is called a {\it direct summand} of $A$.
The algebra $A$ is called {\it indecomposable} if there is no direct sum decomposition of $A$ with two non-trivial subalgebras.
$A$ is called a {\it simple} algebra if $A$ is semisimple and indecomposable.
An ideal $I \neq \{0\}$ in $A$ is called {\it minimal} if $I' \subset I$ for an ideal $I'$ of $A$ implies $I' = \{0\}$ or $I' = I$.

Let $V$ be an $\R$-vector space which is equipped with a positive definite symmetric bilinear form $f$.
We set that the vector space $\R \oplus V$ which is a direct sum of a one dimensional space $\R$ with basis $1$ and $V$.
We define a product in $\R \oplus V$ by 
\[(\alpha 1 + \bm{x}) * (\beta 1 + \bm{y}) := (\alpha \beta + f(\bm{x}, \bm{y})) 1 + (\beta \bm{x} + \alpha \bm{y})\]
for $\alpha, \beta \in \R, \bm{x},\bm{y} \in V$.
Then $(\R \oplus V, *)$ is a formally real Jordan algebra and denoted by $\R \oplus_f V$.
It is well known that $\R \oplus_f V$ is simple if $\dim V > 1$.

\begin{thm} \cite[Theorem 11 in Chapter III]{minnesota}
	Every semisimple Jordan algebra $A$ is a direct sum of simple Jordan algebras.
	These direct summands are exactly the minimal ideals of $A$.
	Two decompositions of a semisimple Jordan algebra into direct sums of simple Jordan algebras are equal up to a permutation of the summands.
\end{thm}

The classification of simple formally real Jordan algebras over $\R$ was given by Jordan, von Neumann and Wigner (see \cite{JNW}).
Let $\C$, $\mathbb{H}$ and $\mathbb{O}$ denote the field of complex numbers, the quaternion algebra and the octonion algebra, respectively.
Let $M_n(\mathcal{D})$ denote the set of all $n \times n$ matrices with entries in $\mathcal{D}$.
Then every simple formally real Jordan algebra is isomorphic to one of the following:
\begin{enumerate}
  \item $\R$,
  \item $\R \oplus_f V (\dim V \geq 2)$,
  \item $H_n(\mathcal{D}) := \{A \in M_n(\mathcal{D}) : {}^t \bar{A} = A\} (n \geq 3)$, where $\mathcal{D} = \R, \C, \mathbb{H}$ and $\bar{A}$ is the standard conjugation in $\mathcal{D}$,
  \item $H_3(\mathbb{O})$.
\end{enumerate}

\begin{remark}
  Since $H_3(\mathbb{O})$ is not special, it does not occur as a simple component of the adjacency Jordan algebra of any Jordan scheme.
\end{remark}

If $A$ is semisimple, then $A$ has the unit element $1$.
An element $e \in A$ is called an \emph{idempotent} if $e * e = e$ and $e \neq 0$.
A subset $I \neq \{0\}$ of $A$ is an ideal in $A$ if and only if there is an idempotent $e$ in $Z(A)$ such that $I = e * A$.
Let $A = I_{1} \oplus \cdots \oplus I_{r}$ be a direct sum decomposition and let $I_{i} = e_{i} * A$, where $e_{i} * e_{i} = e_{i} \in Z(A)$.
The unit element $1$ of $A$ is given by $1 = e_{1} + \cdots + e_{r}$, and $e_{i} * e_{j} = \delta_{i,j} e_{i}$.

For a subset $T \subset M_n(\C)$, we define $\mathrm{Alg}(T)$ the smallest associative algebra containing $T$ and $\mathrm{WL}(T)$ the smallest coherent algebra containing $T$ (we refer to \cite{MRK} on coherent algebras).
Also we define $\mathrm{Sym} : M_n(\C) \to M_n(\C)$ by $\mathrm{Sym}(a) = a + {}^t a$.

We set $\C S := \sum_{i = 0}^{d} \C A_i$ as a $\C$-vector space for a Jordan scheme $(X, S)$.
By definition, we have the following inclusions:
\[\C S \subset \mathrm{Alg}(\C S) \subset \mathrm{WL}(\C S), \quad \C S \subset \mathrm{Sym}(\mathrm{Alg}(\C S)) \subset \mathrm{Sym}(\mathrm{WL}(\C S)).\]

\begin{lem}
  A Jordan scheme $(X, S)$ is an association scheme if and only if $\C S = \mathrm{Alg}(\C S)$.
\end{lem}

\begin{lemma} \cite[Proposition 1.6]{MRK} \label{improper}
  A Jordan scheme $(X, S)$ is improper if and only if $\C S = \mathrm{Sym}(\mathrm{WL}(\C S))$.
\end{lemma}

\begin{corollary}
  A Jordan scheme $(X, S)$ is proper if $\C S \neq \mathrm{Sym}(\mathrm{Alg}(\C S))$.
\end{corollary}

\section{A construction of a Jordan scheme from an elementary abelian $2$-group and a $\{1,-1\}$-matrix satisfying a specified condition}
In this section, we construct a new Jordan scheme from an elementary abelian $2$-group and  a $\{1,-1\}$-matrix satisfying a specified condition.
In the next section, we will prove that the order of such a matrix is limited to $2,4$ or $8$.

Let $I_{2}$ and $J_{2}$ denote the identity matrix and the all-one matrix of order $2$, respectively, and set $K_{2} = J_{2} - I_{2}$.

Let $G$ be an elementary abelian $2$-group of rank $n$.
We assume that there exists a $\{1, -1\}$-matrix $M$ indexed by $G$ with the following condition; for any $x,y,z \in G$ with $x \neq y$,
\[ M_{x, xz} M_{x,yz} M_{y,xz} M_{y,yz} = -1.\]
\[
\bordermatrix{
 &&xz&&yz &  \cr
 && \vdots && \vdots &  \cr
x & \cdots & * & \cdots & * & \cdots  \cr
 &  & \vdots &  & \vdots  &  \cr
 y&\cdots& * &\cdots& * & \cdots  \cr
&&\vdots && \vdots &  \cr
}
\]

We set $H := C_2 \times G$, where $C_2 := \langle h \rangle$ is the group of order $2$.
Then we define $3 |G|+1$ matrices of degree $4|G|$ as follows; for $(s, t) \in H = C_2 \times G$,
\[\sigma_{(s,t)} := \sum_{g \in G} E_{(1,g),(s,gt)} \otimes L_{g,gt} + \sum_{g \in G} E_{(h,g),(hs,gt)} \otimes L_{gt, g}, \]
where $E_{(s,t),(u,v)}$ denotes the matrix unit for $(s,t),(u,v) \in H$ in $M_{H}(\C)$ and for any $x,y \in G$, $L_{x,y}$ is defined by
\begin{align*}
  L_{x,y} := \begin{cases}
    I_2 & \text{if $s=1, t= 1$,} \\
    J_2 & \text{if $s=1, t \neq 1$,} \\
    I_2 & \text{if $s \neq 1$ and $M_{x,y} = 1$,} \\
    K_2 & \text{if $s \neq 1$ and $M_{x,y} = -1$.} \\
  \end{cases}
\end{align*}
And we define $\sigma_{(s,t)}'$ as the expression in which $L_{g,gt}$ and $L_{gt,g}$ in the definition of $\sigma_{(s,t)}$ are replaced with $L_{g,gt}K_2$ and $L_{gt,g}K_2$, respectively.
We remark that $\sigma_{(1,t)} = \sigma_{(1,t)}'$ if $t \neq 1$.

We set $S := \{\sigma_{(1,1)}, \sigma_{(1,1)}'\} \cup \{\sigma_{(1,t)}\}_{1 \neq t \in G} \cup \{\sigma_{(h,t)}, \sigma_{(h,t)}' \}_{t \in G}$.
\begin{thm}
$S$ is a Jordan scheme.
\end{thm}
\begin{proof}
  It is easy to check that $\sigma_{(1,1)}$ is the identity matrix and the sum of all matrices in $S$ is equal to the all-one matrix.
  By direct culculation, we have that
  \begin{align*}
    \sigma_{(1,1)}' * \sigma_{(1,1)}' &= \sigma_{(1,1)}, &\sigma_{(1,1)}' * \sigma_{(1,t)} &= \sigma_{(1,t)},\\
    \sigma_{(1,1)}' * \sigma_{(h,t)} &= \sigma_{(h,t)}', &\sigma_{(1,1)}' * \sigma_{(h,t)}' &= \sigma_{(h,t)},\\
    \sigma_{(1,t)} * \sigma_{(1,v)} &= \begin{cases} 2 \sigma_{(1,1)} + 2 \sigma_{(1,1)}'  & \text{ if $t=v$,} \\ 2 \sigma_{(1,tv)} & \text{ if $t \neq v$,}\end{cases} &\sigma_{(1,t)} * \sigma_{(h,v)} &= \sigma_{(h,tv)} + \sigma_{(h,tv)}',\\
    \sigma_{(1,t)} * \sigma_{(h,v)}' &= \sigma_{(h,tv)} + \sigma_{(h,tv)}'.
  \end{align*} 
  Moreover, it holds that
  \begin{equation*}
    \begin{split}
    2\sigma_{(h,t)} * \sigma_{(h,v)} &= \left( \sum_{k \in G} E_{(1,k),(h,kt)} \otimes L_{k,kt} + \sum_{k \in G} E_{(h,k),(1,kt)} \otimes L_{kt,k}\right) \\ 
    &\qquad \times \left( \sum_{\ell \in G} E_{(1,\ell),(h,\ell v)} \otimes L_{\ell, \ell v} + \sum_{\ell \in G} E_{(h,\ell),(1, \ell v)} \otimes L_{\ell v, \ell}\right) \\
    &\quad + \left( \sum_{\ell \in G} E_{(1,\ell),(h,\ell v)} \otimes L_{\ell, \ell v} + \sum_{\ell \in G} E_{(h,\ell),(1, \ell v)} \otimes L_{\ell v, \ell}\right) \\
    &\qquad \times \left( \sum_{k \in G} E_{(1,k),(h,k t)} \otimes L_{k, k t} + \sum_{k \in G} E_{(h,k),(1, k t)} \otimes L_{kt,k}\right) \\
    &= \sum_{k \in G} E_{(1,k),(1,ktv)} \otimes L_{k,kt} L_{ktv, kt} + \sum_{k \in G} E_{(h,k),(h,ktv)} \otimes L_{kt,k} L_{kt,ktv} \\
    &\qquad + \sum_{k \in G} E_{(1,k),(1,ktv)} \otimes L_{k,kv} L_{ktv, kv} + \sum_{k \in G} E_{(h,k),(h,ktv)} \otimes L_{kv,k} L_{kv,ktv} \\
    &= \sum_{k \in G} E_{(1,k),(1,ktv)} \otimes \left( L_{k,kt} L_{ktv, kt} + L_{k,kv} L_{ktv, kv} \right) \\
    &\qquad + \sum_{k \in G} E_{(h,k),(h,ktv)} \otimes \left( L_{kt,k} L_{kt,ktv} + L_{kv,k} L_{kv,ktv} \right) .
  \end{split}
  \end{equation*}
  If $t = v$, we know that
  \begin{align*}
 2 \sigma_{(h,t)} * \sigma_{(h,v)} &= \sum_{k \in G} E_{(1,k),(1,k)} \otimes \left( L_{k,kt} L_{k, kt} + L_{k,kt} L_{k, kt} \right) + \sum_{k \in G} E_{(h,k),(h,k)} \otimes \left( L_{kt,k} L_{kt,k} + L_{kt,k} L_{kt,k} \right) \\
  &=2 \sigma_{(1,1)}.
  \end{align*}
  If $t \neq v$, considering the definition of $L_{x,y}$, both of the numbers of $I_2$ and $K_2$ in $\{L_{k,kt}, L_{ktv,kt}, L_{k,kv}, L_{ktv,kv}\}$ are odd.
  Similarly, both of the numbers of $I_2$ and $K_2$ in $\{L_{kt,k}, L_{kt,ktv}, L_{kv,k}, L_{kv,ktv}\}$ are also odd. 
  Therefore we obtain that
  \begin{align*}
      2 \sigma_{(h,t)} * \sigma_{(h,v)} &= \sum_{k \in G} E_{(1,k),(1,ktv)} \otimes \left( I_2 + K_2 \right) + \sum_{k \in G} E_{(h,k),(h,ktv)} \otimes \left( I_2 + K_2 \right) \\ &= \sigma_{(1,tv)} + \sigma_{(1,tv)}'.
  \end{align*}
  Thus, we have that 
  \[\sigma_{(h,t)} * \sigma_{(h,v)} = \begin{cases} \sigma_{(1,1)} & \text{ if $t=v$,} \\ \frac{1}{2} \sigma_{(1,tv)} + \frac{1}{2} \sigma_{(1,tv)}' & \text{ if $t \neq v$.} \end{cases}\\    \]
    Similarly, we can show that
  \begin{align*}
    \sigma_{(h,t)} * \sigma_{(h,v)}' &= \begin{cases} \sigma_{(1,1)}' & \text{ if $t=v$,} \\ \frac{1}{2} \sigma_{(1,tv)} + \frac{1}{2} \sigma_{(1,tv)}' & \text{ if $t \neq v$,} \end{cases}\\
    \sigma_{(h,t)}' * \sigma_{(h,v)}' &= \begin{cases} \sigma_{(1,1)} & \text{ if $t=v$,} \\ \frac{1}{2} \sigma_{(1,tv)} + \frac{1}{2} \sigma_{(1,tv)}' & \text{ if $t \neq v$.} \end{cases}
  \end{align*}
\end{proof}

\section{The existence of some $\{1,-1\}$-matrices with the specified condition}
In this section, we prove that the order of a $\{1,-1\}$-matrix $M$, considering in the previous section, is limited to $2, 4$ or $8$.
We begin by recalling the definition of the matrix $M$.
Let $G$ be an elementary abelian $2$-group of rank $n$.
We assume that there exists a $\{1, -1\}$-matrix $M$ with the following condition indexed by $G$; for any $x,y,z \in G$ with $x \neq y$, 
\[M_{x, xz} M_{x,yz} M_{y,xz} M_{y,yz} = -1.\]

For any $c \in G$, we define a polynomial $f_c$ by
\[f_c := \sum_{g \in G} M_{g, gc} ~ x_g y_{gc}.\]

Then by direct calculation, we have that
\begin{align*}
  \sum_{c \in G} {f_c}^2 &= \sum_{c \in G} \left( \sum_{g \in G} M_{g, gc} ~x_g y_{gc} \right) \left( \sum_{h \in G} M_{h, hc} ~x_h y_{hc} \right) \\
  &=\sum_{c, g, h \in G} M_{g, gc} M_{h,hc} ~x_g x_h y_{gc} y_{hc} \\
  &=\sum_{g, c \in G} {M_{g,gc}}^2~ {x_g}^2 ~{y_{gc}}^2 + \sum_{c, g, h \in G, g \neq h} M_{g, gc} M_{h,hc} ~x_g x_h y_{gc} y_{hc} \\
  &=\sum_{g, h \in G} {x_g}^2~ {y_h}^2 + \sum_{c, g, h \in G, g \neq h} M_{g, gc} M_{h,hc} ~x_g x_h y_{gc} y_{hc}.
\end{align*} 

We consider the second summation in the above line.
We fix two different elements $g, h \in G$.
For any $c \in G$, there exists the unique element $c' \in G$ such that $gc' = hc$.
We know $c \neq c'$ because of $g \neq h$ and it is easy to check $(c')'=c$.
Then it follows that $x_g x_h y_{gc} y_{hc} = x_g x_h y_{gc'} y_{hc'}$ and
\[ M_{g,gc} M_{h,hc} + M_{g,gc'} M_{h,hc'} = M_{g,gc} M_{h,hc} + M_{g,hc} M_{h,gc} = 0.\]
Therefore we obtain that
\[\left( \sum_{g \in G} {x_g}^2 \right) \left( \sum_{h \in G} {y_h}^2 \right) = \sum_{c \in G} {f_c}^2.\]

Then, by the Hurwitz Theorem below, we know that $n = 2,4,$ or $8$ and the matrix $M$ represents the sign pattern of the multiplication table for some basis of $\mathbb{C}, \mathbb{H},$ or $\mathbb{O}$.

\begin{thm} \cite[10.1 Theorem]{huppert} (the Hurwitz Theorem)
Suppose that $K$ is a field of characteristic 0.
Suppose there exist polynomials
\[f_i = \sum_{j,k=1}^{n} a_{ijk} x_j y_k \in K[x_1,\dots, x_n, y_1, \dots, y_n],\]
bilinear in the $x_j, y_k$, such that
\[\sum_{i=1}^n {x_i}^2 \sum_{i=1}^n {y_i}^2 = \sum_{i=1}^n {f_i}^2.\]
Then $n$ is one of the numbers $1,2,4, 8$.
\end{thm}

\begin{example}
The following three matrices are examples of possible matrices of orders $2, 4, 8$, respectively.
\[\begin{pmatrix}
  1 & 1 \\ 1 & -1
\end{pmatrix}, \begin{pmatrix}
  1 & 1 & 1 & 1\\
  1 & -1 & 1 & -1 \\
  1 & -1 & -1 & 1 \\
  1 & 1 & -1 & -1 
\end{pmatrix}, \begin{pmatrix}
  1 & 1 & 1 & 1 & 1 & 1 & 1 & 1 \\
  1 & -1 & 1 & -1 & 1 & -1 & -1 & 1 \\
  1 & -1 & -1 & 1 & 1 & 1 & -1 & -1 \\
  1 & 1 & -1 & -1 & 1 & -1 & 1 & -1 \\
  1 & -1 & -1 & -1 & -1 & 1 & 1 & 1 \\
  1 & 1 & -1 & 1 & -1 & -1 & -1 & 1 \\
  1 & 1 & 1 & -1 & -1 & 1 & -1 & -1 \\
  1 & -1 & 1 & 1 & -1 & -1 & 1 & -1 
\end{pmatrix}.\]
\end{example}

Using the GAP system \cite{GAP4}, we enumerated all admissible matrices for each order and verified that the Jordan schemes derived from them are combinatorially isomorphic.
Consequently, only three essentially distinct Jordan schemes arise, of orders $8, 16, 32$.
The Jordan scheme of order $8$ is improper, whereas those of orders $16$ and $32$ are proper.
We denote the latter two by $J_{16}, J_{32}$, respectively.

\section{The real adjacency Jordan algebras of $J_{16}, J_{32}$}
In this section, we analyze the real adjacency Jordan algebras of Jordan schemes $J_{16}$ and $J_{32}$.

\begin{thm}
	Let $(X, S)$ be a Jordan scheme and let $\R S$ be the adjacency Jordan algebra of $(X,S)$.
	Let $1 = e_{1} + e_{2} + \cdots + e_{r}$ be a central idempotent decomposition of $1$ in the algebra $\mathrm{Alg}(\C S)$.
	If $e_{i}$ is in $\R S$, we have $e_{i} \in Z(\R S)$.
	Especially, if all $e_{i}$ is in $\R S$, then $\R S = \bigoplus_{i=0}^{r} e_{i} * \R S$ is a direct sum decomposition of $\R S$.
\end{thm}
\begin{proof}
  We assume that $e_i \in \R S$. Since $e_i$ is a central idempotent, it satisfies $e_i x = x e_i$ for all $x \in \R S$.
  It is then straightforward to verify that $e_i * (b * x) = b * (e_i * x)$ for all $b, x \in \R S$.
\end{proof}

The relation matrix $R(J_{16})$ of $J_{16}$ is the following, namely $R(J_{16}) := \sum_{i = 0}^{12} i \cdot A_i$.
\[R(J_{16}) = \begin{pmatrix}
  0 & 1 & 2 & 2 & 3 & 3 & 4 & 4 & 5 & 6 & 7 & 10 & 8 & 11 & 9 & 12 \\
  1 & 0 & 2 & 2 & 3 & 3 & 4 & 4 & 6 & 5 & 10 & 7 & 11 & 8 & 12 & 9 \\
  2 & 2 & 0 & 1 & 4 & 4 & 3 & 3 & 7 & 10 & 6 & 5 & 9 & 12 & 11 & 8 \\
  2 & 2 & 1 & 0 & 4 & 4 & 3 & 3 & 10 & 7 & 5 & 6 & 12 & 9 & 8 & 11 \\
  3 & 3 & 4 & 4 & 0 & 1 & 2 & 2 & 8 & 11 & 12 & 9 & 6 & 5 & 7 & 10 \\
  3 & 3 & 4 & 4 & 1 & 0 & 2 & 2 & 11 & 8 & 9 & 12 & 5 & 6 & 10 & 7 \\
  4 & 4 & 3 & 3 & 2 & 2 & 0 & 1 & 9 & 12 & 8 & 11 & 10 & 7 & 6 & 5 \\
  4 & 4 & 3 & 3 & 2 & 2 & 1 & 0 & 12 & 9 & 11 & 8 & 7 & 10 & 5 & 6 \\
  5 & 6 & 7 & 10 & 8 & 11 & 9 & 12 & 0 & 1 & 2 & 2 & 3 & 3 & 4 & 4 \\
  6 & 5 & 10 & 7 & 11 & 8 & 12 & 9 & 1 & 0 & 2 & 2 & 3 & 3 & 4 & 4 \\
  7 & 10 & 6 & 5 & 12 & 9 & 8 & 11 & 2 & 2 & 0 & 1 & 4 & 4 & 3 & 3 \\
  10 & 7 & 5 & 6 & 9 & 12 & 11 & 8 & 2 & 2 & 1 & 0 & 4 & 4 & 3 & 3 \\
  8 & 11 & 9 & 12 & 6 & 5 & 10 & 7 & 3 & 3 & 4 & 4 & 0 & 1 & 2 & 2 \\
  11 & 8 & 12 & 9 & 5 & 6 & 7 & 10 & 3 & 3 & 4 & 4 & 1 & 0 & 2 & 2 \\
  9 & 12 & 11 & 8 & 7 & 10 & 6 & 5 & 4 & 4 & 3 & 3 & 2 & 2 & 0 & 1 \\
  12 & 9 & 8 & 11 & 10 & 7 & 5 & 6 & 4 & 4 & 3 & 3 & 2 & 2 & 1 & 0 
\end{pmatrix}\]

Then we know that $\dim \mathrm{Alg}(\C S) = 24$ and the central idempotents of $\mathrm{Alg}(\C S)$ are
\begin{align*}
  e_1 &= \frac{1}{16} \left( A_0 + A_1 - A_2 - A_3 + A_4 - A_5 -A_6 + A_7 + A_8 - A_9 + A_{10} + A_{11} - A_{12} \right), \\
  e_2 &= \frac{1}{16} \left( A_0 + A_1 - A_2 - A_3 + A_4 + A_5 +A_6 - A_7 - A_8 + A_9 - A_{10} - A_{11} + A_{12} \right), \\
  e_3 &= \frac{1}{16} \left( A_0 + A_1 - A_2 + A_3 - A_4 - A_5 -A_6 + A_7 - A_8 + A_9 + A_{10} - A_{11} + A_{12} \right), \\
  e_4 &= \frac{1}{16} \left( A_0 + A_1 - A_2 + A_3 - A_4 + A_5 +A_6 - A_7 + A_8 - A_9 - A_{10} + A_{11} - A_{12} \right), \\
  e_5 &= \frac{1}{16} \left( A_0 + A_1 + A_2 - A_3 - A_4 - A_5 -A_6 - A_7 + A_8 + A_9 - A_{10} + A_{11} + A_{12} \right), \\
  e_6 &= \frac{1}{16} \left( A_0 + A_1 + A_2 - A_3 - A_4 + A_5 +A_6 + A_7 - A_8 - A_9 + A_{10} - A_{11} - A_{12} \right), \\
  e_7 &= \frac{1}{16} \left( A_0 + A_1 + A_2 + A_3 + A_4 - A_5 -A_6 - A_7 - A_8 - A_9 - A_{10} - A_{11} - A_{12} \right), \\
  e_8 &= \frac{1}{16} \left( A_0 + A_1 + A_2 + A_3 + A_4 + A_5 +A_6 + A_7 + A_8 + A_9 + A_{10} + A_{11} + A_{12} \right), \\
  e_9 &= \frac{1}{2} \left( A_0 - A_1 \right)
\end{align*}
and they are in $\R S$.
Moreover we have that
\[\dim_{\R} e_i * \R S = 1 \quad (i = 1, 2, \dots, 8), \quad \dim_{\R} e_9 * \R S = 5.\]
A basis of $e_9 * \R S$ is $\{e_9, e_9 A_5, e_9 A_7, e_9 A_8, e_9 A_{9}  \}$ and the Jordan products of two of them are
\[ e_9 * (e_9 A_i) = (e_9 A_i) * e_9 = e_9 A_i, \quad (e_9 A_i) * (e_9 A_j) = \delta_{i,j} e_9 \quad (i,j = 5,7,8,9). \]
Thus, $e_9 * \R S \cong \R \oplus_f \R^4$, where $f(\bm{x},\bm{y}) = \sum_i x_i y_i$ for any $\bm{x} = (x_1,x_2,x_3,x_4), \bm{y} = (y_1,y_2,y_3,y_4)$.

\begin{thm}
  Let $(X, S)$ be the Jordan scheme $J_{16}$. 
  \[ \R S \cong 8 \R \oplus (\R \oplus_f \R^4).\]
\end{thm}

Similarly, we can obtain the following.

\begin{thm}
  Let $(X, S)$ be the Jordan scheme $J_{32}$.
  \[ \R S \cong 16 \R \oplus (\R \oplus_f \R^8),\]
  where $f(\bm{x},\bm{y}) = \sum_{i} x_i y_i$ for any $\bm{x} = (x_1, x_2, \dots, x_8), \bm{y} = (y_1, y_2, \dots, y_8)$.
\end{thm}

\section*{Acknowledgment}
The first author would like to thank Jesse Lansdown for valuable discussions on Jordan schemes at Shinshu University in August 2024.
The first author was supported by JSPS KAKENHI Grant Number JP22K03266.
The second author was supported by JSPS KAKENHI Grant Number JP20K03557.

\end{document}